\documentclass[11pt]{article}
\usepackage{amsmath}
\usepackage{amssymb}
\usepackage{amsthm}
\usepackage[usenames]{color}
\usepackage{amscd}
\usepackage{indentfirst}
\usepackage[colorlinks=true,linkcolor=webgreen,filecolor=webbrown,
citecolor=webgreen]{hyperref}
\definecolor{webgreen}{rgb}{0,.5,0}
\definecolor{webbrown}{rgb}{.6,0,0}

\def\N{{\mathbb{N}}}
\def\Z{{\mathbb{Z}}}
\def\1{{\bf 1}}

\def\id{\operatorname{id}}

\newtheorem{teor}{Theorem}
\newtheorem{cor}{Corollary}

\newtheorem{prop}{Proposition}

\newtheorem{lem}{Lemma}

\begin{document}

\title{{\bf Counting invertible sums of squares modulo $n$ and a new generalization of Euler's totient function}}
\author{Catalina Calder\'{o}n, Jos\'{e} Mar\'{\i}a Grau, Antonio M. Oller-Marc\'{e}n, and L\'{a}szl\'{o} T\'{o}th}
\maketitle

\begin{abstract} In this paper we introduce and study a family $\Phi_k$ of arithmetic functions generalizing Euler's totient function.
These functions are given by the number of solutions to the equation $\gcd(x_1^2+\ldots +x_k^2, n)=1$ with $x_1,\ldots,x_k \in {\Z}/n{\Z}$
which, for $k=2,4$ and $8$ coincide, respectively, with the number of units in the rings of Gaussian integers, quaternions and octonions
over ${\Z}/n{\Z}$. We prove that $\Phi_k$ is multiplicative for every $k$, we obtain an explicit formula for $\Phi_k(n)$ in
terms of the prime-power decomposition of $n$ and derive an asymptotic formula for $\sum_{n\le x} \Phi_k(n)$. As a tool we
investigate the multiplicative arithmetic function that counts the number of solutions to $x_1^2+\ldots +x_k^2\equiv \lambda$ (mod $n$) for
$\lambda$ coprime to $n$, thus extending an old result that dealt only with the prime $n$ case.
\end{abstract}




\section{Introduction}

Euler's totient function $\varphi$ is one of the most famous arithmetic functions used in number theory. Recall that $\varphi(n)$ is
defined as the number of positive integers less than or equal to $n$ that are coprime to $n$. Many generalizations and analogs of Euler's
function are known. See, for instance \cite{freed2012,gen1,dani,gen2,gen0,gen3} or the special chapter on this topic in \cite{gen4}. Among
the generalizations, the most significant is probably the Jordan's totient function $\mathbf{J}_k$ given by $\mathbf{J}_k(n)= n^k
\prod_{p\mid n}(1-p^{-k})$ ($n\in \N:=\{1,2,\ldots\}$). See \cite{and}, \cite[pp.\ 147--155]{jor}, \cite{van}.

In this paper we introduce and study a new generalization of $\varphi$. In particular, given $k \in \mathbb{N}$ we define
\begin{equation} \label{def_phi_k}
\Phi_k(n):=\textrm{card}\ \{(x_1,\ldots,x_k)\in(\mathbb{Z}/n\mathbb{Z})^k : \gcd(x_1^2+\ldots+x_k^2,n)=1\}.
\end{equation}

Clearly, $\Phi_1(n)=\varphi(n)$ and it is the order of the group of units of the ring $\mathbb{Z}/n\mathbb{Z}$. On the other hand, $\Phi_2(n)$
is the restriction to the set of positive integers of the Euler function defined on the Gaussian integers $\Z[i]$. Thus $\Phi_2(n)$,
denoted also by $\operatorname{GIphi}(n)$ in the literature, computes the number of Gaussian integers in a reduced residue system modulo $n$. See
\cite{Cross1983}. In the same way, $\Phi_4(n)$ and $\Phi_8(n)$ compute, respectively, the number of invertible quaternions and octonions
over $\mathbb{Z}/n\mathbb{Z}$.

In order to study the function $\Phi_k$ we need to focus on the functions
\begin{equation} \label{def_rho}
\rho_{k,\lambda}(n):=\textrm{card}\ \{(x_1,\dots,x_k)\in(\mathbb{Z}/n\mathbb{Z})^k : x_1^2+\ldots +x_k^2\equiv\lambda\pmod{n}\}
\end{equation}
which count the number of points on hyperspheres in $(\mathbb{Z}/n\mathbb{Z})^k$ and, in particular, in the case $\gcd(\lambda,n)=1$.
These functions were already studied in the case when $n$ is an odd prime by V.~H.~Lebesgue in 1837. In particular he proved the following
result (\cite[Chapter X]{jor2}).

\begin{prop}
\label{LEB}
Let $p$ be an odd prime and let $k,\lambda$ be positive integers with $p\nmid \lambda$. Put $t=(-1)^{(p-1)(k-1)/4}p^{(k-1)/2}$
and $\ell=(-1)^{k(p-1)/4}p^{(k-2)/2}$. Then
\begin{equation*}
\rho_{k,\lambda}(p)=\begin{cases} p^{k-1}+t, & \textrm{if $k$ is odd and $\lambda$ is a quadratic residue modulo $p$};\\ p^{k-1}-t,
& \textrm{if $k$ is odd and $\lambda$ is a not quadratic residue modulo $p$};\\ p^{k-1}-\ell, & \textrm{if $k$ is even}.
\end{cases}
\end{equation*}
\end{prop}

The paper is organized as follows. First of all, in Section \ref{S2} we study the values of $\rho_{k,\lambda}(n)$ in the case
$\gcd(\lambda,n)=1$, thus generalizing Lebesgue's work. In Section \ref{S3} we study the functions $\Phi_k$, in particular we prove
that they are multiplicative and we give a closed formula for $\Phi_k(n)$ in terms of the prime-power decomposition of $n$. Section
\ref{S4} is devoted to deduce an asymptotic formula for $\sum_{n\le x} \Phi_k(n)$. Finally,
we close our work suggesting some ideas that leave the door open for future work.


\section{Counting points on hyperspheres (mod $n$)}
\label{S2}

Due to the Chinese Remainder Theorem, the function $\rho_{k,\lambda}$, defined by \eqref{def_rho} is multiplicative; i.e.,
if $n=p_1^{r_1}\cdots p_m^{r_m}$,
then $\rho_{k,\lambda}(n)=\rho_{k,\lambda}(p_1^{r_1})\cdots\rho_{k,\lambda}(p_m^{r_m})$. Hence, we can restrict ourselves to the case when
$n=p^s$ is a prime-power. Moreover, since in this paper we focus on the case $\gcd(\lambda,n)=1$, we will always assume that $p\nmid \lambda$.
The following result will allow us to extend Lebesgue's work to the odd prime-power case.

\begin{lem}
\label{LEM:REDP}
Let $p$ be an odd prime and let $s\in \N$. If $p\nmid\lambda$, then
\begin{equation*}
\rho_{k,\lambda}(p^s)=p^{(s-1)(k-1)}\rho_{k,\lambda}(p).
\end{equation*}
\end{lem}

\begin{proof}
It is easily seen that any solution to the congruence $x_1^2+\ldots +x_k^2\equiv \lambda\pmod{p^{s+1}}$ must be of the form
$(a_1+t_1p^s,\ldots,a_k+t_kp^s)$, where $0\leq t_1,\ldots, t_k\leq p-1$, for some $(a_1,\ldots,a_k)$ such that
$a_1^2+\ldots + a_k^2\equiv \lambda\pmod{p^s}$. Now, $(a_1+t_1p^s)^2+\ldots+(a_k+t_kp^s)^2\equiv \lambda\pmod{p^{s+1}}$ if
and only if $2a_1t_1+\ldots +2a_kt_k\equiv -K\pmod{p}$, where $K$ is such that $a_1^2+\ldots +a_k^2=Kp^s+\lambda$.
Since $a_i\not\equiv 0\pmod{p}$ for some $i\in\{1,\dots,k\}$, it follows that there are exactly $p^{k-1}$ possibilities for
$(t_1,\dots,t_k)$. We obtain that $\rho_{k,\lambda}(p^{s+1})= p^{k-1}\rho_{k,\lambda}(p^s)$, and the result follows inductively.
\end{proof}

If $p=2$ we have a similar result.

\begin{lem}
\label{LEM:RED2}
Let $s\geq 3$ and let $\lambda\in \N$ be odd. Then,
\begin{equation*}
\rho_{k,\lambda}(2^{s})=2^{(s-3)(k-1)}\rho_{k,\lambda}(8).
\end{equation*}
\end{lem}

\begin{proof} In the case $p=2$ the proof of Lemma \ref{LEM:REDP} does not work since $2a_1t_1+\ldots +2a_kt_k\equiv -K \pmod{2}$ holds
only if $K$ is even. Therefore, we use that every solution of the congruence $x_1^2+\ldots+x_k^2\equiv \lambda \pmod{2^{s+1}}$ is of the
form $(a_1+t_12^{s-1},\ldots,a_k+t_k2^{s-1})$, where $0\leq t_1,\ldots, t_k\leq 3$, for some $(a_1,\ldots,a_k)$ satisfying
$a_1^2+\ldots +a_k^2\equiv \lambda \pmod{2^{s-1}}$, that is $a_1^2+\ldots +a_k^2= L 2^{s-1} +\lambda$ with an integer $L$. Now, taking
into account that $s\geq 3$, $(a_1+t_12^{s-1})^2+\ldots +(a_k+t_k2^{s-1})^2\equiv \lambda \pmod{2^{s+1}}$ if and only if
\begin{equation} \label{lemma_2_eq_1}
2(a_1t_1+\ldots + a_kt_k)\equiv -L \pmod{4}.
\end{equation}

Here the condition \eqref{lemma_2_eq_1} holds true if and only if $L$ is even, i.e.,
$a_1^2+\ldots +a_k^2\equiv \lambda \pmod{2^s}$. Hence we need the solutions $(a_1,\ldots,a_k)$ of the congruence (mod $2^s$), but only
those satisfying
\begin{equation} \label{lemma_2_eq_2}
0\leq a_1,\ldots, a_k< 2^{s-1}
\end{equation}

It is easy to see that their number is $\rho_{k,\lambda}(2^{s})/2^k$, since all
solutions of the congruence (mod $2^s$) are $(a_1+u_12^{s-1},\ldots, a_k+u_k2^{s-1})$ with $(a_1,\ldots,a_k)$ verifying \eqref{lemma_2_eq_2} and
$0\leq u_1,\ldots,u_k\leq 1$. Since $a_i$ must be odd for some $i\in\{1,\ldots,k\}$, for a fixed even $L$, \eqref{lemma_2_eq_1} has $2\cdot 4^{k-1}$
solutions $(t_1,\ldots,t_k)$. We deduce that $\rho_{k,\lambda}(2^{s+1})= 2\cdot 4^{k-1} \rho_{k,\lambda}(2^{s})/2^k = 2^{k-1}\rho_{k,\lambda}(2^s)$.
Now the result follows inductively on $s$.
\end{proof}

As we have just seen, unlike when $p$ is an odd prime, the recurrence is now based on $\rho_{k,\lambda}(2^3)$. Hence, the cases $s=1,2,3$; i.e.,
$n=2,4,8$, must be studied separately. In order to do so, the following general result will be useful.

\begin{lem}
\label{LEM:REC}
Let $k,\lambda$ and $n$ be positive integers. Then
\begin{equation*}
\rho_{k,\lambda}(n)=\sum_{\ell=0}^{n-1}\rho_{1,\ell}(n)\rho_{k-1,\lambda-\ell}(n).
\end{equation*}
\end{lem}

\begin{proof}
Let $(x_1,\dots,x_k) \in (\mathbb{Z}/n\mathbb{Z})^k$ be such that $x_1^2+\cdots+ x_k^2\equiv \lambda\pmod{n}$. Then, for some $\ell \in\{0,\ldots,n-1\}$ we
have that $x_1^2\equiv \ell \pmod{n}$ and $x_2^2+\ldots+ x_k^2\equiv \lambda-\ell \pmod{n}$ and hence the result.
\end{proof}

Now, given $k,n\in \N$ let us define the matrix $M(n)=\left(\rho_{1,i-j}(n)\right)_{0\leq i,j\leq n-1}$. If we consider the column vector
$R_{k}(n)=\left(\rho_{k,i}(n) \right)_{0\leq i\leq n-1}$, then Lemma \ref{LEM:REC} leads to the following recurrence relation:
$$R_{k}(n)=M(n)\cdot R_{k-1}(n).$$

In the following proposition we use this recurrence relation to compute $\rho_{k,\lambda}(2^s)$ for $s=1,2,3$ and odd $\lambda$.

\begin{lem}
\label{PROP2}
Let $k$ be a positive integer. Then
\begin{itemize}
\item[i)] $\rho_{k,1}(2)=2^{k-1}$,
\item[ii)] $\rho_{k,1}(4) = 4^{k-1}+2^{\frac{3 k}{2}-1} \sin \left(\frac{\pi  k}{4}\right)$,
\item[iii)] $\rho_{k,3}(4) = 4^{k-1}-2^{\frac{3 k}{2}-1} \sin \left(\frac{\pi  k}{4}\right)$,
\item[iv)] $\rho_{k,1}(8)=  2^{2 k-3} \left(2^k+2^{\frac{k}{2}+1} \sin \left(\frac{\pi  k}{4}\right)+2 \sin \left(\frac{1}{4} \pi(k+1)\right)-2 \cos \left(\frac{1}{4} \pi (3k+1)\right)\right)$,
\item[v)] $\rho_{k,3}(8)=2^{2 k-3} \left(2^k-2^{\frac{k}{2}+1} \sin \left(\frac{\pi  k}{4}\right)-2 \left(\cos \left(\frac{1}{4} \pi (k+1)\right)+\cos \left(\frac{3}{4} \pi  (k+1)\right)\right)\right)$,
\item[vi)] $\rho_{k,5}(8)=2^{2 k-3} \left(2^k+2^{\frac{k}{2}+1} \sin \left(\frac{\pi  k}{4}\right)-2 \sin \left(\frac{1}{4} \pi (k+1)\right)+2 \cos \left(\frac{1}{4}\pi (3k+1)\right)\right)$,
\item[vii)] $\rho_{k,7}(8)= 2^{2 k-3} \left(2^k-2^{\frac{k}{2}+1} \sin \left(\frac{\pi  k}{4}\right)-2 \sin \left(\frac{1}{4} (3 \pi  k+\pi)\right)+2 \cos \left(\frac{1}{4} \pi  (k+1)\right)\right)$.
\end{itemize}
\end{lem}

\begin{proof}
First of all, observe that
\begin{equation*}
M(2)=\begin{pmatrix} 1&1\\ 1&1\end{pmatrix},\ M(4)=\begin{pmatrix} 2 & 0 & 0 & 2 \\
 2 & 2 & 0 & 0 \\
 0 & 2 & 2 & 0 \\
 0 & 0 & 2 & 2 \\\end{pmatrix},\ M(8)=\begin{pmatrix}
 2 & 0 & 0 & 0 & 2 & 0 & 0 & 4 \\
 4 & 2 & 0 & 0 & 0 & 2 & 0 & 0 \\
 0 & 4 & 2 & 0 & 0 & 0 & 2 & 0 \\
 0 & 0 & 4 & 2 & 0 & 0 & 0 & 2 \\
 2 & 0 & 0 & 4 & 2 & 0 & 0 & 0 \\
 0 & 2 & 0 & 0 & 4 & 2 & 0 & 0 \\
 0 & 0 & 2 & 0 & 0 & 4 & 2 & 0 \\
 0 & 0 & 0 & 2 & 0 & 0 & 4 & 2 \\
\end{pmatrix}.
\end{equation*}

Let us compute ii). We know that $R_k(4)=M(4)\cdot R_{k-1}(4)$. Hence, since the eigenvalues of $M(4)$ are $\{4,2+2 i,2-2 i,0\}$, we know that
\begin{equation*}
\rho_{k,1}(4)=C_1 4^k+C_2 (2+2i)^k+C_3(2-2i)^k.
\end{equation*}

In order to compute $C_1$, $C_2$ and $C_3$ it is enough to observe that $\rho_{1,1}(4)=2$, $\rho_{2,1}(4)=8$ and $\rho_{3,1}(4)=24$. Hence
\begin{align*}&4 C_1 + (2 + 2 i) C_2 + (2 - 2 i) C_3 = 2,\\ &16 C_1 + 8 i C_2 - 8 i C_3 = 8,\\ &64 C_1 - (16 - 16 i) C_2 - (16 + 16 i)
C_3 = 24. \end{align*}

We deduce
\begin{equation*}
\rho_{k,1}(4)=\frac{1}{4} \left( 4^k- i (2+2 i)^k+i (2-2 i)^k\right)=2^{2 k-2}+2^{\frac{3 k}{2}-1} \sin \left(\frac{\pi  k}{4}\right),
\end{equation*}
as claimed.

To compute the other cases note that the eigenvalues of $M(2)$ are $\{0,2\}$ while the eigenvalues of $M(8)$ are
$$\left\{8,4+4 i,4-4 i,\sqrt{2}(-2-2 i),\sqrt{2}(2+2 i),\sqrt{2}(-2+2 i),\sqrt{2}(2-2 i),0\right\}.$$ Thus, in each case we only
need to compute the corresponding initial conditions and constants. The final results have been obtained with the help of
Mathematica ``ComplexExpand'' command.
\end{proof}

Note that a different approach to compute the values $\rho_{k,\lambda}(n)$, using the Gauss quadratic sum was given in \cite{Toth14}.


\section{Counting invertible sums of squares (mod $n$)}
\label{S3}

Given positive integers $k,n$, this section is devoted to computing $\Phi_k(n)$, defined  by \eqref{def_phi_k}. Let $A(k,\lambda,n)$ denote the
set of solutions $(x_1,\dots,x_k)\in (\Z/n\Z)^k$ of the congruence $x_1^2+\ldots+x_k^2\equiv\lambda\pmod{n}$. First of all, let us define the set
\begin{equation*}
\mathcal{A}_k(n):=\bigcup_{\substack{1\leq \lambda\leq n\\ \gcd(\lambda,n)=1}} A(k,\lambda,n).
\end{equation*}

Hence, $\Phi_k(n)=\textrm{card}\ \mathcal{A}_k(n)$ and, since the union is clearly disjoint, it follows that
\begin{equation*}
\Phi_k(n)=\sum_{\substack{1\leq \lambda\leq n\\ \gcd(\lambda,n)=1}}\rho_{k,\lambda}(n).
\end{equation*}

The following result shows the multiplicativity of $\Phi_k$ for every positive $k$.

\begin{prop}
Let $k$ be a positive integer. Then $\Phi_k$ is multiplicative; i.e., $\Phi_k(mn)=\Phi_k(m)\Phi_k(n)$ for every
$m,n\in \N$ such that $\gcd(m,n)=1$.
\end{prop}

\begin{proof}
Let us define a map $F: \mathcal{A}_k(m) \times \mathcal{A}_k(n) \longrightarrow \mathcal{A}_k(mn)$ by
\begin{equation*}
F((a_1,\ldots,a_k),(b_1,\ldots,b_k))=(na_1+mb_1,\dots,na_k+mb_k).
\end{equation*}

Note that if $(a_1,\ldots,a_k)\in \mathcal{A}_k(m)$, then $a_1^2+\ldots+a_k^2 \equiv \lambda_1 \pmod m$ for some $\lambda_1$ with
$\gcd(\lambda_1,m)=1$. In the same way, if $(b_1,\ldots,b_k) \in \mathcal{A}_k(n)$, then $b_1^2+\ldots+b_k^2 \equiv \lambda_2 \pmod n$
for some $\lambda_2$ with $\gcd(\lambda_2,n)=1$. Consequently,
\begin{align*}
(na_1+mb_1)^2+\ldots+(na_k+mb_k)^2&=n^2 (a_1^2+\ldots+a_k^2)+m^2(b_1^2+\ldots+b_k^2)\\& +2mn(b_1a_1+\ldots+b_ka_k)\equiv\\ &\equiv n^2\lambda_1+m^2\lambda_2\pmod {mn}.
\end{align*}

Since it is clear that $\gcd(n^2 \lambda_1+m^2 \lambda_2,mn)=1$, it follows that $(na_1+mb_1,\ldots,na_k+mb_k)\in\mathcal{A}_k(mn)$ and thus
$F$ is well-defined.

Now, let $(c_1,\dots,c_k) \in \mathcal{A}_k(mn)$. Then $c_1^2+\ldots+c_k^2 \equiv \lambda \pmod {mn}$ for some $\lambda$ such that
$\gcd(\lambda,mn)=1$. Let us define $a_i\equiv c_i\pmod{m}$ and $b_i\equiv c_i\pmod{n}$ for every $i=1,\dots,k$. It follows that $(a_1,\dots,a_k)\in\mathcal{A}_k(m)$, $(b_1,\dots,b_k)\in\mathcal{A}_k(n)$ and, moreover, $F((a_1,\dots,a_k),(b_1,\dots,b_k))=(c_1,\dots,c_k)$.
Hence, $F$ is surjective.

Finally, assume that $$(na_1+mb_1,\ldots,na_k+mb_k) \equiv (n\alpha_1+m\beta_1,\ldots,n\alpha_k+m\beta_k) \pmod {mn}$$ for some
$(a_1,\ldots,a_k), (\alpha_1,\ldots,\alpha_k) \in \mathcal{A}_k(m)$ and for some $(b_1,\dots,b_k),(\beta_1,\dots,\beta_k)\in \mathcal{A}_k(n)$.
Then, for every $i=1,\dots,k$ we have that $na_i+mb_i \equiv  n\alpha_i+m \beta_i \pmod {mn}$. From this, it follows that
$a_i\equiv\alpha_i\pmod{m}$ and that $b_i\equiv\beta_i\pmod{n}$ for every $i$ and hence $F$ is injective.

Thus, we have proved that $F$ is bijective and the result follows.
\end{proof}

Since we know that $\Phi_k$ is multiplicative, we just need to compute its values over prime-powers. We do so in the following result.

\begin{prop}
\label{PP}
Let $k,r$ be positive integers. Then
\begin{itemize}
\item[i)] $\Phi_k(2^r)=\varphi(2^{kr})=2^{kr-1}$.
\item[ii)] If $p$ is an odd prime,
$$ \Phi_k(p^r)= \varphi(p^{kr}) - (-1)^{k(p-1)/4} \varphi(p^{kr-k/2})= p^{kr-\frac{k}{2}-1}(p-1)\left(p^{k/2}-(-1)^{k(p-1)/4}\right).
$$
\end{itemize}
\end{prop}

\begin{proof}
\
\begin{itemize}
\item[i)] If $r=1,2,3$ the result readily follows from Lemma \ref{PROP2} by a simple computation. Now, if $r>3$ we can apply Lemma \ref{LEM:RED2} to obtain that
\begin{align*}
\Phi_k(2^r)&=\sum_{\substack{1\leq i\leq 2^r\\ 2\nmid i}}\rho_{k,i}(2^r)=2^{(r-3)(k-1)}\sum_{\substack{1\leq i\leq 2^r\\ 2\nmid i}}\rho_{k,i}(8)&\\&=2^{(r-3)(k-1)}\sum_{j=0}^{2^{r-3}-1}\sum_{\substack{8j+1\leq i\leq 8(j+1)-1\\ 2\nmid i}}\rho_{k,i}(8)=\\&=2^{(r-3)(k-1)}2^{r-3}\sum_{\substack{1\leq i\leq 7\\ 2\nmid i}}\rho_{k,i}(8)=\\&=2^{(r-3)(k-1)}2^{r-3}2^{3k-1}=2^{rk-1}=\varphi(2^{kr}).
\end{align*}

\item[ii)] Due to Lemma \ref{LEM:REDP} it can be seen, as is the previous case, that
$$\Phi_k(p^r)=p^{k(r-1)}\sum_{i=1}^{p-1}\rho_{k,i}(p).$$
Thus, it is enough to apply Proposition \ref{LEB}.
\end{itemize}
\end{proof}

Finally, we summarize the previous work in the following result.

\begin{teor}
\label{Theorem_1}
Let $k$ be a positive integer. Then the function $\Phi_k$ is multiplicative and for every $n\in \N$,
\begin{equation*}
\Phi_k(n)=\begin{cases} n^{k-1} \varphi(n), & \textrm{if $k$ is odd}; \\
n^{k-1} \varphi(n) \displaystyle \prod_{\substack{p \mid n\\ p>2}}  \left(1-\frac {(-1)^{k(p-1)/4}}{p^{k/2}}\right), &
\textrm{if $k$ is even}.
\end{cases}
\end{equation*}
\end{teor}

Written more explicitly, we deduce that
\begin{equation*}
\Phi_k(n)=\begin{cases} n^{k-1} \varphi(n), & \textrm{if $k$ is odd}; \\ n^{k-1} \varphi(n)
\displaystyle
\prod_{\substack{p \mid n\\ p>2}}  \left(1-\frac {1}{p^{k/2}}\right), & \text{if $k\equiv 0 \, \pmod{4}$}; \\
n^{k-1} \varphi(n)
\displaystyle
\prod_{\substack{p \mid n\\ p \equiv 1 \text{\rm (mod $4$)}}}  \left(1-\frac {1}{p^{k/2}}\right)
\displaystyle
\prod_{\substack{p \mid n\\ p \equiv -1 \text{\rm (mod $4$)}}}  \left(1+\frac {1}{p^{k/2}}\right), &
\text{if $k\equiv 2 \, \pmod{4}$}.
\end{cases}
\end{equation*}

When $k$ is a multiple of $4$, $\Phi_k$ is closely related to $\mathbf{J}_{k/2}$. The following result, which follows from Theorem \ref{Theorem_1}
and the definition of Jordan's totient function $\mathbf{J}_{k}$ makes this relation explicit.

\begin{cor}
Let $k\in \N$ be a multiple of $4$ and let $n\in \N$. Then,
\begin{equation*}
\Phi_k(n)= n^{k/2-1} \mathbf{J}_{k/2}(n) \varphi(n) \frac{2^{k/2}}{2^{k/2}-1+n\text{\rm (mod $2$)}}.
\end{equation*}

Moreover, if $k/4$ is odd, then we have
\begin{equation*}
\frac{\Phi_{k}(n)}{\Phi_{k/4}(n)}=n^{k/4} \mathbf{J}_{k/2}(n) \frac{2^{k/2}}{2^{k/2}-1+n\text{\rm (mod $2$)}}.
\end{equation*}
\end{cor}

Recall that in the case $k=4$, $\Phi_4(n)$ is the number of units in the ring $\mathbb{H}(\mathbb{Z}/n\mathbb{Z})$. If, in addition, $n$
is odd then $\Phi_4(n)=n\mathbf{J}_{2}(n) \varphi(n)$ which is the well-known formula for the number of regular matrices in
the ring $\mathrm{M}_2(\mathbb{Z}/n\mathbb{Z})$. Of course, this is not a surprise since it is known that for an odd $n$ the
rings $\mathbb{H}(\mathbb{Z}/n\mathbb{Z})$ and $\mathrm{M}_2(\mathbb{Z}/n\mathbb{Z})$ are isomorphic (\cite{cel}).

Some elementary properties of $\Phi_k$, well known for Euler's function (the case $k=1$) follow at once by Theorem \ref{Theorem_1}.
For example, we have

\begin{cor} Let $k\in \N$ be fixed.

i) If $m,n\in \N$ such that $n\mid m$, then $\Phi_k(n) \mid \Phi_k(m)$.

ii) Let $m,n\in \N$ and let $d=\gcd(m,n)$. Then
$\Phi_k(mn)\Phi_k(d) = d^k \Phi_k(m)\Phi_k(m)$.

iii) If $n,m\in \N$, then $\Phi_k(n^m)=n^{km-k}\Phi_k(n)$.
\end{cor}


\section{The average order of $\Phi_k(n)$}
\label{S4}

The average order of $\varphi(n)$ is well-known. Namely,
\begin{equation*}
\frac1{x} \sum_{n\leq x}\varphi(n) \sim \frac{3}{\pi^2} x \quad (x\to \infty),
\end{equation*}
see, for example, \cite[Th.\ 330]{asin}. In fact, the best known asymptotic formula is due to Walfisz \cite{wal}:
\begin{equation} \label{walfisz}
\sum_{n\leq x} \varphi(n) =  \frac{3}{\pi^2} x^2 +  O(x (\log x)^{2/3} (\log \log x)^{4/3}).
\end{equation}

We now generalize this result.

\begin{teor}
Let  $k\in \N$ be any fixed integer. Then
\begin{equation*}
\sum_{n\le x }\Phi_k(n)= \frac{C_k}{k+1} x^{k+1} + O(x^k R_k(x)),
\end{equation*}
where
\begin{equation*}
C_k= \frac{6}{\pi^2}, \quad R_k=(\log x)^{2/3} (\log \log x)^{4/3}, \quad \text{if $k$ is odd};
\end{equation*}
\begin{equation*}
C_k=\frac{3}{4}\prod_{p>2} \left(1- \frac1{p^2} - \frac{(-1)^{k(p-1)/4}(p-1)}{p^{k/2+2}}\right), \quad R_k(x)=\log x, \quad
\text{if $k$ is even}.
\end{equation*}
\end{teor}

\begin{proof} If $k$ is odd, then this result follows easily by partial summation from the fact that $\Phi_k(n)=n^{k-1} \varphi (n)$ using
Walfisz' formula \eqref{walfisz}.

Assume now that $k\in \N$ is even. Since the function $\phi_k$ is multiplicative, we deduce by the Euler product formula that
\begin{equation*}
\sum_{n=1}^{\infty} \frac{\Phi_k(n)}{n^s} = \zeta(s-k) G_k(s),
\end{equation*}
where
\begin{equation*}
G_k(s)= \left(1-\frac1{2^{s-k+1}}\right) \prod_{p>2} \left(1- \frac1{p^{s-k+1}}
- \frac{(-1)^{k(p-1)/4}(p-1)}{p^{s-k/2+1}} \right)
\end{equation*}
is absolutely convergent for $\Re s> k$. This shows that $\Phi_k=\id_k*g_k$ in terms of the Dirichlet convolution, where
$\id_k(n)=n^k$ ($n\in \N$) and the multiplicative function
$g_k$ is defined by
\begin{equation*}
g_k(p^r)= \begin{cases} -2^{k-1}, & \text{if $p=2$, $r=1$}; \\
-p^{k-1} - (-1)^{k(p-1)/4}p^{k/2-1}(p-1) , & \text{if $p>2$, $r=1$}; \\ 0, & \text{otherwise}.
\end{cases}
\end{equation*}

We obtain
\begin{equation*}
\sum_{n\le x} \Phi_k(n)=  \sum_{de\le x} g_k(d)e^k = \sum_{d\le x} g_k(d) \sum_{e\le x} e^k = \sum_{d\le x} g_k(d) \left(\frac{(x/d)^{k+1}}{k+1}
+O((x/d)^k) \right)
\end{equation*}
\begin{equation} \label{error}
=  \frac{x^{k+1}}{k+1}  G_k(k+1) + O\left(x^{k+1} \sum_{d>x} \frac{|g_k(d)|}{d^{k+1}}\right) +
O\left(x^k \sum_{d\le x} \frac{|g_k(d)|}{d^k}\right).
\end{equation}

Here for every $k\ge 4$ we have
\begin{equation*}
\sum_{d\le x} \frac{|g_k(d)|}{d^k} \le \prod_{p\le x}
\sum_{r=0}^{\infty} \frac{|g_k(p^r)|}{p^{kr}} = \prod_{p\le x}
\left(1+ \frac{|g_k(p)|}{p^k}\right) \ll \prod_{p\le x} \left(1+
\frac{p^{k-1}+p^{k/2-1}+p^{k/2}}{p^k} \right)
\end{equation*}
\begin{equation*}
<  \prod_{p\le x} \sum_{r=0}^{\infty} \frac1{p^r} =  \prod_{p\le x} \left(1-\frac1{p}\right)^{-1} \ll \log x
\end{equation*}
by Mertens's theorem. In the case $k=2$ this gives
\begin{equation*}
\sum_{d\le x} \frac{|g_2(d)|}{d^2} \ll \prod_{\substack{p\le x\\ p\equiv 1 \text{(mod $4$)}}} \left(1+ \frac{2}{p}-\frac1{p^2} \right)
\prod_{\substack{p\le x\\ p\equiv -1 \text{(mod $4$)}}} \left(1+\frac1{p^2} \right)
\end{equation*}
\begin{equation*}
\ll  \prod_{\substack{p\le x\\ p\equiv 1 \text{(mod $4$)}}} \left(1+ \frac1{p}\right)^2 \ll \log x,
\end{equation*}
using that
\begin{equation*}
\prod_{\substack{p\le x\\ p\equiv 1 \text{(mod $4$)}}} \left(1- \frac1{p}\right) \sim c (\log x)^{-1/2},
\end{equation*}
with a certain constant $c$, cf. \cite{Uch1971}. Hence the last last error term of \eqref{error} is $O(x^k\log x)$ for every $k\ge 2$ even.

Furthermore, note that for every $k\ge 2$, $|g_k(n)|\le n^{k/2}
\sigma_{k/2-1}(n)$ ($n\in \N$), where $\sigma_t(n)=\sum_{d\mid n}
d^t$. Using that $\sigma_t(n)<\zeta(t)n^t$ for $t>1$ we conclude
that $|h_k(n)|\ll n^{k-1}$ for $k\ge 4$. Therefore,
\begin{equation*}
\sum_{d>x} \frac{|g_k(d)|}{d^{k+1}} \ll \sum_{d>x} \frac1{d^2} \ll \frac1{x}.
\end{equation*}

In the case $k=2$, using that $\sigma_1(n)\ll n\log n$ (this suffices) we have $h_2(n)\ll n^3\log n$ and
\begin{equation*}
\sum_{d>x} \frac{|g_2(d)|}{d^3} \ll \sum_{d>x} \frac{\log d}{d^2} \ll \frac{\log x}{x}.
\end{equation*}

Hence the first error term of \eqref{error} is $O(x^k)$ for $k\ge 4$ and it is $O(x^2\log x)$ for $k=2$. This completes the proof.
\end{proof}

\begin{cor} {\rm ($k=2,4$)}
\begin{equation*}
\sum_{n\le x }\Phi_2(n)= \frac{x^3}{4} \prod_{p\equiv 1 \text{\rm (mod $4$)}} \left(1-\frac{2}{p^2} +\frac1{p^3}\right) \prod_{p\equiv -1
\text{\rm (mod $4$)}} \left(1-\frac1{p^3}\right) + O(x^2 \log x),
\end{equation*}
\begin{equation*}
\sum_{n\le x }\Phi_4(n)= \frac{3x^5}{20} \prod_{p>2} \left(1-\frac1{p^2} -\frac1{p^3}+\frac1{p^4}\right) + O(x^4 \log x).
\end{equation*}
\end{cor}


\section{Conclusions and further work}
\label{S5}

The generalization of $\varphi$ that we have presented in this paper is possibly one of the closest to the original idea which consists of counting
units in a ring. In addition, both the elementary and asymptotic properties of $\Phi_k$ extend those of $\varphi$ in a very natural way.
There are many other results regarding $\varphi$ that have not been considered here but that,
nevertheless, may have their extension to $\Phi_k$. For instance, in 1965 P. Kesava Menon \cite{menon} proved the following identity:
\begin{equation*}
\sum_{\substack{1\le j\le n \\ \gcd(j,n)=1}} \gcd(j-1,n)= \varphi(n)
d(n),
\end{equation*}
valid for every $n\in \N$, where $d(n)$ denotes the number of
divisors of $n$. This identity has been generalized in several ways.
See, for example \cite{hau1,hau2,tan,Toth2011}. Also,
\begin{equation*}
\sum_{\substack{1\le j\le n \\ \gcd(j,n)=1}} \gcd(j^2-1,n)=
\varphi(n)h(n),
\end{equation*}
where $h$ is a multiplicative function given explicitly in
\cite[Cor.\ 15]{Toth2011}. Our work suggests the following
generalization:

\begin{equation*}
\sum_{\substack{1\leq x_1,\ldots,x_k\leq n\\ \gcd(x_1^2+\ldots +
x_k^2,n)=1}} \gcd(x_1^2+\ldots +x_k^2-1,n)=\Phi_k(n) \Psi_k(n),
\end{equation*}
where $\Psi_k$ is a multiplicative function to be found.

Another question is on minimal order. As well known (\cite[Th.\ 328]{asin}), the minimal order of $\varphi(n)$ is $e^{-\gamma} n(\log \log n)^{-1}$,
where $\gamma$
is Euler's constant, that is
$$
\liminf_{n\to \infty} \frac{\varphi(n)\log\log n}{n}= e^{-\gamma}.
$$

It turns out by Theorem \ref{Theorem_1} that for every $k\in \N$ odd,
$$
\liminf_{n\to \infty} \frac{\Phi_k(n)\log\log n}{n^k}= e^{-\gamma}.
$$

Find the minimal order of $\Phi_k(n)$ in the case when $k$ is even.


\medskip

\noindent C. Calder\'{o}n \\
Departamento de Matem\'{a}ticas, Universidad del Pa\'{i}s Vasco \\
Facultad de Ciencia y Tecnolog\'{i}a \\
Barrio Sarriena, s/n, 48940 Leioa, Spain \\
{mtpcagac@lg.ehu.es}
\medskip

\noindent J. M. Grau \\
Departamento de Matem\'{a}ticas, Universidad de Oviedo \\
Avda. Calvo Sotelo, s/n, 33007 Oviedo, Spain \\
{grau@uniovi.es}
\medskip

\noindent A. M. Oller-Marc\'{e}n \\
Centro Universitario de la Defensa \\
Ctra. de Huesca, s/n, 50090 Zaragoza, Spain \\
{oller@unizar.es}
\medskip

\noindent L. T\'{o}th \\
Department of Mathematics, University of P\'{e}cs \\
Ifj\'us\'ag u. 6, H-7624 P\'ecs, Hungary \\
{ltoth@gamma.ttk.pte.hu}

\end{document}